\documentclass[reqno,11pt]{amsart}
\usepackage{graphicx} % Required for inserting images
\usepackage{amsmath,tikz, tikz-cd}
\usepackage[utf8]{inputenc}
\usepackage{hyperref}
\usepackage[normalem]{ulem}
\usepackage{amsmath}
\usepackage{amssymb}
\usepackage{amsthm}
\usepackage{mathtools}
\usepackage{verbatim}
\usepackage{xcolor}
\usepackage[skip=3pt, indent=20pt]{parskip}
\usepackage{caption}
\usepackage{comment}
\usepackage{marginnote}
\usepackage{cite}
\usepackage{color}
\usepackage{adjustbox}
\newtheorem{thmA}{Theorem}%[section]

\newtheorem{theorem}{Theorem}[section]
\newtheorem{lemma}[theorem]{Lemma}
\newtheorem{proposition}[theorem]{Proposition}
\newtheorem{corollary}[theorem]{Corollary}

\theoremstyle{definition}
\newtheorem{definition}{Definition}

\theoremstyle{remark}
\newtheorem{remark}[theorem]{Remark}
\newcommand{\newword}[1]{\emph{\textbf{#1}}}

\newcommand{\wt}{\ensuremath\mathrm{wt}}
\newcommand{\B}{\ensuremath\mathcal{B}}
\newcommand{\A}{\ensuremath\mathcal{A}}
\newcommand{\rex}{\operatorname{rex}}

\newcommand{\g}{\mathfrak{g}}

\newcommand{\I}{\mathcal{I}}

\newcommand{\ZZ}{\mathbb{Z}}

\newcommand{\E}{\mathcal{E}}
\newcommand{\F}{\mathcal{F}}

\usepackage{mathtools}

\DeclarePairedDelimiter{\floor}{\lfloor}{\rfloor}

\input xy
\usepackage[all]{xy}
\xyoption{line}
\xyoption{arrow}
\xyoption{color}
\SelectTips{cm}{}
\usepackage{tikz}
\usetikzlibrary{decorations.markings}
\usetikzlibrary{decorations.pathreplacing}

\tikzcdset{scale cd/.style={every label/.append style={scale=#1},
    cells={nodes={scale=#1}}}}

\usepackage{xcolor}
\usepackage[colorinlistoftodos]{todonotes}

\title{A Local Characterization of Unions of Demazure crystals}

\author[Assaf]{Sami Assaf}
\address{Department of Mathematics, University of Southern California, 3620 S. Vermont Ave., Los Angeles, CA 90089-2532, U.S.A.}
\email{shassaf@usc.edu}
\thanks{S.A. supported by NSF grant DMS-2246785 and Simons Foundation Award 953878.}

\author[Gonz\'{a}lez]{Nicolle Gonz\'{a}lez}
\address{Department of Mathematics, The University of British Columbia, Vancouver, Canada}
\email{nicolle@math.ubc.ca}
\date{\today}

\begin{document}

\begin{abstract}
  We characterize subsets of highest weight $\g$-crystals that arise as unions of Demazure crystals, for any symmetrizable Kac-Moody Lie algebra $\g$. We provide a local characterization for these subsets and prove they admit disjoint decompositions into Demazure atoms. As a consequence, we give a new characterization for when a subset of a highest weight crystal is a Demazure crystal
  as well as a crystal-theoretic proof that any Polo module admits a relative Schubert filtration.
\end{abstract}

\maketitle
\section{Introduction}

For $\g$ a symmetrizable Kac-Moody Lie algebra, the integrable irreducible highest weight $U(\g)$-modules $V(\lambda)$ are indexed by dominant integral weights $\lambda\in P^+$. Kashiwara \cite{Kas91} introduced crystal bases as set-theoretic abstractions of Lusztig's geometric canonical bases \cite{Lus90}. The \emph{highest weight crystal} $\B(\lambda)$ associated to $V(\lambda)$ is a combinatorial skeleton that contains important information such as the character and the irreducible decomposition.

For $w$ an element of the Weyl group $W$ and $\mathfrak{b} \subset \g$ a Borel subalgebra, the \emph{Demazure module} $V_w(\lambda)$ is the $U(\mathfrak{b})$-submodule of $V(\lambda)$ generated by the extremal weight vector of $V(\lambda)$ with weight $w \lambda$ \cite{Dem74}. 
Geometrically, Demazure modules arise 
as the dual of the space of global sections $\Gamma(X_w,\mathcal{L}_\lambda)$ of a suitable line bundle $\mathcal{L}_\lambda$ over the Schubert variety $X_w$.
The Demazure character formula, proved in various degrees of generality in \cite{Dem74a, Jos85, Andersen, Kumar-Demazure, Lit95, Kas93, Mathieu-Demazure}, generalizes the Weyl-Kac character formula. For classical types, Littelmann \cite{Lit95} proved the existence of a subset $\B_w(\lambda) \subseteq \B(\lambda)$ whose character corresponds to that of the Demazure module $V_w(\lambda)$. Kashiwara \cite{Kas93} extended the definition of Demazure crystals by lifting the Demazure operators to give an explicit construction for $\B_w(\lambda)$.  
Demazure crystals present similar advantages as highest weight crystals and, despite being widely studied \cite{besson, AG21,ADG, Assaf2,AG21b,blasiak2,spellman,wen,besson2,hong,Kou20,kouno2, Naoi12,Naoi12}, their structure is less developed.

In this paper, we characterize when a subset of a highest weight crystal is a Demazure crystal. Our main result can be stated as follows:
\begin{thmA}\label{thm:A}
     A subset $X$ of a highest weight crystal $\B$ is a Demazure crystal iff $X$ satisfies the following three conditions:
     \begin{itemize}
         \item[(E)] $X$ is \newword{extremal} if $X$ is nonempty and for any $i$-string $S$ of $\B$, $S \cap X$ is either $\varnothing$ or $S$ or $\{b\}$, where $e_i(b)=0$;
         \item[(I)] $X$ is \newword{ideal} if it is extremal and for extremal elements $x,y \in X$ with $x = e_{j}^* e_{i_1}^* \cdots e_{i_m}^* (y)$, we have $ f_{i_m}^* \cdots f_{i_1}^*(x) \in X \sqcup \{0\}$;
         \item[(P)] $X$ is \newword{principal} if it is extremal and for every connected component $X' \subseteq X$, there exists an extremal element $x \in X'$ such that for all extremal $y\in X'$, we have $\wt(y) \preceq \wt(x)$.
     \end{itemize}
\end{thmA}

An immediate application of this characterization is that one can establish the Demazure nonnegativity of polynomials using Demazure crystals, parallel to the use of highest weight crystal to establish Schur positivity. Finding nonnegative expansions of various families of polynomials into Demazure characters is an active area of research \cite{Blasiak,per, Assaf,BucimasScrimshaw, miller,Assaf-weak,Gib21}.

Conversely, we can also use Demazure nonnegativity to establish whether a subset of a highest weight crystal whose character is a Demazure character is a Demazure crystal. This is especially useful since both the extremal and ideal conditions are local.

\begin{thmA}\label{thm:B}
    For a subset $X$ of a highest weight crystal whose character is a Demazure character, $X$ is a Demazure crystal if and only if $X$ is ideal.
\end{thmA}    

\emph{Extremal} subsets of a crystal, defined by condition (E) in Theorem~\ref{thm:A}, were originally defined by Kashiwara \cite{Kas94} and recently developed further by the authors and Dranowski \cite{ADG}. To establish Theorems \ref{thm:A} and \ref{thm:B} we study a large subclass of extremal subsets of highest weight crystals which we term \emph{ideal} and make up the heart of the paper. Ideal subsets are characterized by the local property (I) in Theorem \ref{thm:A}, and determined by a lower order ideal $\I$ of $W$ under strong Bruhat order. Ideal subsets are also geometrically well-motivated as their characters coincide with those of \emph{dual Polo modules}.   
Polo modules arise in the work of Polo, Mathieu, Joseph, Kumar, van der Kallen and others in their study of excellent filtrations of Borel modules; see \cite{vdk, Pol89, Mat89, Jos85, Jos03, Kumar-Demazure, Kumar-book}.
The \emph{dual Polo module} $V_{\I}(\lambda)$, is obtained by replacing the single Schubert variety in the dual Demazure module $\Gamma(X_w,\mathcal{L}_\lambda)$ with a \emph{union of Schubert varieties} $\Gamma(\bigcup_{w \in \I} X_w,\mathcal{L}_\lambda)$ and then dualizing. 
 In Theorem \ref{thm:ideal} we show that any ideal subset of $\B(\lambda)$ must be of the form $\B_\I(\lambda) := \bigcup_{w \in \I} \B_w(\lambda)$, and thus its character coincides with that of $V_\I(\lambda)$. As a consequence, we obtain a completely local characterization of unions of Demazure crystals. 
\begin{thmA}\label{thm:C}
    A subset $X$ of a highest weight crystal $\B$ is ideal if and only if it is a (nonempty) union of Demazure crystals.
\end{thmA} 

Joseph \cite{Jos85} raised the question of when the tensor product of Demazure modules admits an \emph{excellent filtration}, in which successive quotients are Demazure modules. While Mathieu \cite{Mat89} showed this does not happen in general (see also \cite{Kou20, ADG}), Polo \cite{Pol89} refined the question to conjecture that the tensor product of Demazure modules admits a \emph{Schubert filtration}, in which successive quotients are Polo modules. Our Theorem~\ref{thm:C} establishes a local crystal-theoretic condition for when a module admits a Schubert filtration, thus providing a new tool for approaching Polo's conjecture. 

In a similar vein, we also relate ideal subsets to crystals for \emph{minimal relative Schubert modules}. 
Geometrically, these arise as the kernel of the restriction $\Gamma(X_w,\mathcal{L}_\lambda) \rightarrow \Gamma(\partial X_w,\mathcal{L}_\lambda)$, where $\partial X_w = \bigcup_{u \prec w} X_u$ denotes the \emph{boundary} of a Schubert variety $X_w$. A \emph{Demazure atom} is a subset of a crystal whose character coincides with the dual of the minimal relative Schubert module. In Theorem \ref{thm:atom-positive} we prove that all ideal subsets are disjoint unions of Demazure atoms. This decomposition is a combinatorial manifestation of the fact that a Borel module with a \emph{Schubert filtration}
necessarily also admits a \emph{relative Schubert filtration}, in which successive quotients are minimal relative Schubert modules; for finite types this was proven geometrically in \cite{vdk} and crystal-theoretically in \cite{Armon25}. 

In light of Theorem~\ref{thm:C}, an ideal subset $\B_\I(\lambda)$ is a single Demazure crystal precisely when the defining ideal $\I$ is principal. Culminating, in Theorem \ref{thm:demazure} we show how principality of the ideal translates to condition (P) above, thus yielding Theorems \ref{thm:A} and \ref{thm:B}.

The paper is structured as follows. In \S2, we establish definitions and notation for highest weight crystals, which we generalize to Demazure crystals in \S3. In \S4, we review extremal crystals and establish results relating paths in the crystal between extremal weights with Bruhat order. In \S5, we present our main results concerning ideal subsets, proving Theorem \ref{thm:C}. We then discuss atomic decompositions of ideal subsets in \S6, before returning to Demazure crystals in \S7 where we establish Theorems \ref{thm:A} and \ref{thm:B}.

\section{Highest weight crystals} Let $\g$ be any symmetrizable Kac-Moody Lie algebra and denote by $I$ the vertex set of the Dynkin diagram of $\g$, $P$ the weight lattice, $\{ \alpha_i \;|\; i \in I\} \subset P$ the simple roots, $\{ \alpha_i^\vee \;|\; i \in I\} \subset P^\vee$ the coroots, and $P^+= \{ \lambda \in P \;|\; \langle \alpha_i^\vee, \lambda \rangle \in \ZZ_{\geq 0} \;\forall i\}$ the set of dominant integral weights.

\begin{definition} A \newword{$\g$-crystal} is a set $\B$ endowed with the maps 
\[
  \wt : \B \rightarrow P, \hspace{1em}
  \varepsilon_i,\varphi_i : \B \rightarrow \mathbb{Z}, \hspace{1em}
  e_i,f_i  : \B \rightarrow \B \sqcup \{0\}  
\] 
subject to the following relations for all $i\in I$ and all $b,b'\in \B$,
\begin{itemize}
\item $\varphi_i(b)  = \langle \alpha^{\vee}_i , \wt(b) \rangle + \varepsilon_i(b)$; 
\item$b' = e_i(b)$ if and only if $b = f_i(b')$, in which case 
\[
\wt(b') = \wt(b) + \alpha_i, \quad \varepsilon_i(b) = \varepsilon_i(b') +1, \quad \varphi_i(b') = \varphi_i(b)+1
\]
\item $\varepsilon_i(b) = \max\{k \mid e_i^k(b)\in \B\}$, and
  $\varphi_i(b) = \max\{k \mid f_i^k(b)\in \B\}$.
\end{itemize}
\end{definition}
The \newword{crystal operators} $\{e_i,f_i\}_{i\in I}$ are also referred to as the raising and lowering operators, respectively. 

For any $i \in I$, an \newword{i-string} is a subset $S \subset \B$ of the form 
\[ S = \{ e_i^m(b) \neq 0, f_i^n(b) \neq 0 \;|\; m,n \geq 0\} \subset \B\]
for some $b \in \B$. Thus, for any $b \in \B$, we may consider the \emph{head} and \emph{tail} of its $i$-string and set: 
\[
  e_i^*(b)  =  e_i^{\varepsilon_i(b)}(b), \qquad \text{and} \qquad f_i^*(b)  = f_i^{\varphi_i(b)}(b) .
\]

Recall that an element $b \in \B$ is called a \newword{highest weight vector} (resp. \newword{lowest weight vector}) if $e_i(b) =0$ (resp. $f_i(b)=0$) for all $i \in I$.

The set of irreducible integrable highest weight $U(\g)$-modules are indexed by the set of dominant integral weights. Thus, for each $\lambda \in P^+$, let $\B(\lambda)$ denote the $\g$-crystal corresponding to the crystal basis of the irreducible $U(\g)$-module with highest weight $\lambda$.  Thus, $\B(\lambda)$ is connected with highest weight vector $b_\lambda$ such that $\wt(b_\lambda) = \lambda$. 

%%%%%%%%%%%%%%%%%%%%%%%%%%%%%%%%%%%%%%%%%%%%%%%%%%%%%%%%%%%%
%  Demazure
%%%%%%%%%%%%%%%%%%%%%%%%%%%%%%%%%%%%%%%%%%%%%%%%%%%%%%%%%%%%
\section{Demazure Crystals} 

The \newword{Weyl group} $W$ of $\g$ is generated by simple reflection $\{s_i \;|\; i \in I\}$. For any $w \in W$, an expression $w = s_{i_1}\dots s_{i_n}$ is \newword{reduced} if $n$ is minimal among all possible expressions for $w$. We write $\rex(w)$ for the set of all reduced expressions for $w$. 
For $w\in W$ and $s_{i_1} \cdots s_{i_{\ell}}\in \rex(w)$, define
\begin{equation*} \label{eq:F-definition}
\E_w = \bigcup_{m_i \in \mathbb{N}} \left\{e_{i_1}^{m_1} e_{i_2}^{m_2} \cdots e_{i_{\ell}}^{m_{\ell}}\right\}
\qquad \text{and}\qquad
  \F_w = \bigcup_{m_i \in \mathbb{N}} \left\{f_{i_1}^{m_1} f_{i_2}^{m_2} \cdots f_{i_{\ell}}^{m_{\ell}}\right\}.
\end{equation*} 
Regarded as subsets of mappings on $\B(\lambda)$, both $\E_w$ and $\F_w$ are independent of the choice of reduced expression \cite[\S2.10]{Jos03}. Thus, for any $X \subseteq B(\lambda)$ and $w \in W$, the sets $\E_w(X)$ and $\F_w(X)$ are well-defined.

\begin{definition}[\cite{Kas93}] \label{def:Dem}
  For $\lambda\in P^+$ and $w\in W$, the \newword{Demazure crystal}  $\B_w(\lambda)$ is
  \begin{equation} \label{eq:Dem-definition}
    \B_w(\lambda) = \F_{w} \{ b_\lambda \} .
  \end{equation}
\end{definition}

\begin{theorem}[\cite{Kas93}]\label{thm:kas-dem}
  The Demazure crystal $\B_w(\lambda)$ is well-defined and its character is the character of the Demazure module $V_{w}(\lambda)$.
\end{theorem}

The Weyl group $W$ is equipped with a \emph{length function} $\ell: W \rightarrow \mathbb{N}$ and \emph{Bruhat} partial order given by transitive closure of relations $u \prec u t$ for $t$ a reflection with $\ell(u) \le \ell(ut)$. We will often invoke the following Subword Property \cite[Theorem~2.2.2]{BB05}.

\begin{lemma}[\cite{BB05}]
    Let $s_{j_n} \cdots s_{j_1} \in \rex(w)$. Then $u \preceq w$ if and only if there exists $1 \le i_1 < \cdots < i_m \le n$ such that $s_{j_{i_m}} \cdots s_{j_{i_1}} \in \rex(u)$. 
\end{lemma}

Note that the converse of the property is not true. That is, when $u \preceq w$ it is not the case that every reduced expression for $u$ appears as a subword of some reduced expression for $w$.

It is important to note that different pairs $(\lambda, w)$ and $(\lambda,w')$ may generate the same Demazure crystal.
Namely, for each $\lambda \in P^+$ consider  $W_\lambda$, the stabilizer subgroup of $\lambda$ in $W$, and denote by $\floor{w}^\lambda$ 
the minimal 
length coset representative of $wW_\lambda$. Then, for any $v \in wW_\lambda$ we have $\B_w(\lambda) = \B_v(\lambda)$. Furthermore, containment of Demazure modules is as follows \cite[Prop.~3.4]{ADG}.

\begin{lemma}[\cite{ADG}]\label{lem:elements}
    Let $u,w \in W$ and $\lambda\in P^+$. Then $\B_u(\lambda) \subseteq \B_w(\lambda)$ if and only if $\floor{u}^\lambda \preceq w$. 
\end{lemma}

%%%%%%%%%%%%%%%%%%%%%%%%%%%%%%%%%%%%%%%%%%%%%%%%%%%%%%%%%%%%
%  Extremal
%%%%%%%%%%%%%%%%%%%%%%%%%%%%%%%%%%%%%%%%%%%%%%%%%%%%%%%%%%%%
\section{Extremal crystals}

For each $\lambda \in P^+$, the set of \newword{extremal weights} $\mathcal{O}(\lambda)$ is the orbit of $\lambda$ under the action of $W$. Thus, an element $b \in \B(\lambda)$ is an \newword{extremal weight vector} if $\wt(b) \in \mathcal{O}(\lambda)$, i.e. if $\wt(b) = w\lambda$ for some $w \in W$. In particular, two extremal weights vectors satisfy $u\lambda \leq w \lambda$ in dominance order if and only if $w\preceq u$ in Bruhat order.

Note that extremal weight spaces are one-dimensional, which is to say for each $\alpha\in\mathcal{O}(\lambda)$, there is a unique $b\in\B$ such that $\wt(b)=\alpha$.

\begin{definition}\label{def:extremal}
   A subset $X$ of a $\g$-crystal $\B$ is \newword{extremal} if $X$ is nonempty and for any $i$-string $S$ of $\B$, $S \cap X$ is either $\varnothing$ or $S$ or $\{b\}$, where $e_i(b)=0$. 
\end{definition}

\begin{remark} \label{rem:extremal}
When $\B = \B(\lambda)$ for some $\lambda \in P^+$, if $X\subset \B(\lambda)$ is extremal and $e_i(b) \neq 0$ for some $b \in X$, then $e_i^*(b) \in X$. Iterating, this implies the highest weight vector $b_\lambda$ lies in $X$. In particular, $X \subset \B(\lambda)$ must be connected. Thus, more generally, any extremal subset $X \in \B$ decomposes as a disjoint union of connected components, $ X = \bigsqcup_{\lambda \in \mathcal{J}} X \cap \B(\lambda)$ for some $\mathcal{J} \subset P^+$, each of which is extremal, with $b_\lambda \in X$ for all $\lambda \in \mathcal{J}$. 
\end{remark}

Evidently, if $X \subset \B(\lambda)$ and $Y \subset \B(\mu)$ with $\lambda \neq \mu$ then $X \cap Y = \emptyset$. Thus, generically, we obtain the following (see also \cite[Lemma 4.1]{Armon25}).

\begin{lemma}\label{lem:extremal-union-intersections}
    Unions and nonempty intersections of extremal subsets are extremal.
\end{lemma}

\begin{proof}
Suppose $X$ and $Y$ are extremal subsets. Since both are nonempty, so is $X \cup Y$. Given any $i$-string $S \subset \B(\lambda)$, since both $X$ and $Y$ are extremal and $S \cap (X \cup Y) = (S \cap X) \cup (S \cap Y)$
then $X \cup Y$ is extremal.

Now, assuming $X \cap Y \neq \emptyset$, there must exist $\lambda \in P^+$ such that $b_\lambda \in X \cap Y$. So then given any $i$-string $S \subset \B(\lambda)$ and $S \cap (X \cap Y) = (S \cap X) \cap (S \cap Y)$, then as before we obtain that $X \cap Y$ is extremal.
\end{proof}

Combining \cite[Proposition~3.2.3(ii)]{Kas93} and \cite[Proposition~3.3.5]{Kas93} shows Demazure crystals are extremal subsets.

\begin{proposition}[\cite{Kas93}]\label{prop:extremal}
    For $\lambda\in P^{+}$ and $w\in W$, the subset $\B_w(\lambda) \subseteq \B(\lambda)$ is extremal.
\end{proposition}

By Lemma~\ref{lem:elements}, if we restrict our attention to extremal weight elements, the crystal structure reduces to Bruhat order, provided we take minimal length coset representatives.

This gives a means to establish paths in the crystal between certain extremal weight elements as follows.

\begin{lemma}\label{lem:path}
 Let $x,y \in \B(\lambda)$ be extremal weight vectors with $\wt(x) = u \lambda$ and $\wt(y) = w\lambda$.
 Then 
    $y = f_{j_m}^* \cdots f_{j_1}^*(x)$ if and only if $s_{j_m} \cdots s_{j_1} \in \rex(\floor{wu^{-1}}^{u\lambda})$. In particular, this implies $\wt(x) \ge \wt(y)$.
\end{lemma}

\begin{proof}
We begin by observing that for any $u \preceq w$ we have $\floor{w u^{-1}}^{u\lambda} = \floor{w}^\lambda (\floor{u}^\lambda)^{-1}$.

    Now, proceed by induction on the length of $\floor{w u^{-1}}^{u\lambda}$. If $\floor{w u^{-1}}^{u\lambda}$ has length $0$, then $\floor{w}^\lambda = \floor{u}^\lambda$ and $y = x$, so the result trivially holds. 

    So suppose $u \prec w$ with $\ell(\floor{w u^{-1}}^{u\lambda})=m+1>0$. Assume the result holds for any extremal element of weight $w' \lambda$ with $\ell(\floor{w' u^{-1}}^{u\lambda})\leq m$.  Suppose $y = f_i^* f_{j_m}^* \cdots f_{j_1}^*(x)$ and let $y' = f_{j_m}^* \cdots f_{j_1}^*(x)$, which is extremal of weight $s_i w \lambda$. By induction, $s_{j_m} \cdots s_{j_1} \in \rex(\floor{s_iw u^{-1}}^{u\lambda})$, and thus $s_i s_{j_m} \cdots s_{j_1} \in \rex(\floor{w u^{-1}}^{u\lambda})$.

    Conversely, suppose $s_i s_{j_m} \cdots s_{j_1} \in \rex(\floor{w u^{-1}}^{u\lambda})$. Then $s_{j_m} \cdots s_{j_1} \in \rex(\floor{s_iw u^{-1}}^{u\lambda})$ so by induction, $y' = f_{j_m}^* \cdots f_{j_1}^*(x)$ is the extremal weight vector of weight $s_iw u^{-1}$. Therefore, $y = f_i^* f_{j_m}^* \cdots f_{j_1}^*(x)$ must also be the extremal weight vector of weight $w u^{-1}$.
\end{proof}

%%%%%%%%%%%%%%%%%%%%%%%%%%%%%%%%%%%%%%%%%%%%%%%%%%%%%%%%%%%%
%  Ideal
%%%%%%%%%%%%%%%%%%%%%%%%%%%%%%%%%%%%%%%%%%%%%%%%%%%%%%%%%%%%
\section{Ideal subsets of highest weight crystals} 

In this section we introduce and characterize an intermediate family of subsets which interpolate between Demazure and extremal subsets and coincide with unions of Demazure crystals over lower order ideals. 

\begin{definition}\label{def:ideal}
  An extremal subset $X \subseteq \B$ is \newword{ideal} if for extremal elements $x,y \in X$ with $x = e_{j}^* e_{i_1}^* \cdots e_{i_m}^* (y)$, we have $ f_{i_m}^* \cdots f_{i_1}^*(x) \in X$.
\end{definition}

The ideal condition of Definition~\ref{def:ideal} states that if a sequence of lowering operators stays within $X$, then so does the given subsequence; see Fig.~\ref{fig:Dem4}. 

\begin{figure}[ht]
\begin{tikzcd}[cramped,column sep=small, row sep=1em]
      && x \arrow[dl,red, "f_{j}^*" red,thick,swap]\\
      & \bullet  \arrow[dr, blue, "f_{i_1}^*" blue,thick] &&& \\
       && \bullet \arrow[dr, dashed]&&\\
      &&&\bullet \arrow[dr, teal, "f_{i_m}^*" teal,thick]&\\
      &&&& y  
\end{tikzcd}
$\qquad \Longrightarrow \qquad$
\begin{tikzcd}[cramped,column sep=small, row sep=1em]
      && x \arrow[dl,red, "f_{j}^*" red,thick,swap] \arrow[dr, blue, "f_{i_1}^*" blue,thick] \\
      & \bullet  \arrow[dr, blue, "f_{i_1}^*" blue,thick] &&\bullet \arrow[dr, dashed]& \\
       && \bullet \arrow[dr, dashed]&&\bullet  \arrow[dr, teal, "f_{i_m}^*" teal,thick]\\
      &&&\bullet \arrow[dr, teal, "f_{i_m}^*" teal,thick]&&\bullet\\
      &&&& \bullet  
\end{tikzcd}
\caption{An illustration of the ideal condition: if $x,y \in X$ with $y = f_{i_m}^* \cdots f_{i_1}^* f_{j}^*(x)$, then $f_{i_m}^* \cdots f_{i_1}^*(x)\in X$.}\label{fig:Dem4}
\end{figure}
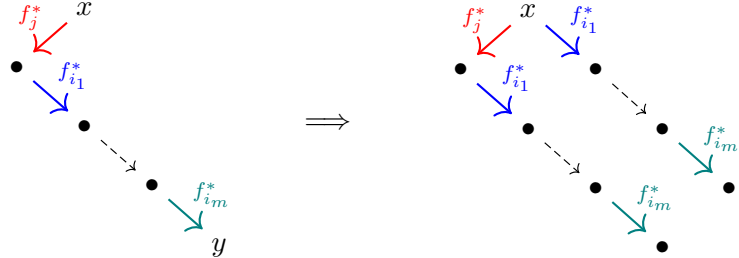

It is easy to see that the ideal condition is closed under unions.

\begin{lemma}\label{lem:ideal-union}
    The union of ideal subsets is an ideal subset.
\end{lemma}

\begin{proof}
By Remark \ref{rem:extremal}, without loss of generality, we may assume $X,Y \subset \B(\lambda)$ for some $\lambda \in P^+$. By Lemma \ref{lem:extremal-union-intersections} we know that $X \cup Y$ is extremal.

Now observe that if both $y$ and $x = e_{j}^* e_{i_1}^* \cdots e_{i_m}^*(y)$ are extremal elements in $X \cup Y$, by symmetry say $y \in X$, then since $X$ is extremal then $x \in X$, and so by the ideal axiom for $X$, $f_{i_m}^* \cdots f_{i_1}^*(x) \in X \subset X \cup Y$. 
\end{proof}

Demazure crystals, by construction, satisfy the ideal condition.

\begin{proposition}\label{prop:ideal}
    For $\lambda\in P^{+}$ and $w\in W$, the subset $\B_w(\lambda) \subseteq \B(\lambda)$ is ideal.
\end{proposition}

\begin{proof}
    By Proposition~\ref{prop:extremal}, $\B_w(\lambda)$ is an extremal subset of $\B(\lambda)$. 
    
    Suppose $x,y \in \B_w(\lambda)$ are extremal with $x = e_{j}^* e_{i_1}^* \cdots e_{i_m}^* (y)$ and weights $\wt(x)=u\lambda$ and $\wt(y)=v\lambda$ for some $u,v \in W$ (see Fig.~\ref{fig:Dem4}). We may assume $u = \floor{u}^\lambda$ and $v = \floor{v}^\lambda$ are minimal length representatives. Since $y \in \B_w(\lambda)$, we have $\B_v(\lambda) \subset \B_w(\lambda)$. Therefore, by Lemma~\ref{lem:elements}, $v = \floor{v}^\lambda \preceq w$. 

    Let $v_0 = s_{i_m} \cdots s_{i_1} s_{j}$, so that $v = v_0 u$ with $\ell(v) = \ell(u) + \ell(v_0)$. Write $v_1 = s_{i_m} \cdots s_{i_1}$. By the Subword Property, we have $v_1 u \prec v_0 u = v \prec w$. In particular, by Lemma~\ref{lem:elements}, $\B_{v_1u}(\lambda) \subset \B_w(\lambda)$. Let  $z = f_{i_m}^* \dots f_{i_1}^* (x)$. Either $z=0$ or $z$ is extremal with weight $v_1u$ and so $z \in \B_{v_1u}(\lambda) \subset \B_w(\lambda)$. Thus, $z \in \B_w(\lambda) \sqcup \{0\}$, and so $\B_w(\lambda)$ is ideal.
\end{proof}

Given an ideal subset $X$, for every extremal element of $X$, the Demazure crystal of that extremal weight must also be contained in $X$.

\begin{lemma}\label{lem:ideal}
    If $X \subseteq \B(\lambda)$ is an ideal subset, and $y \in X$ is an extremal element with $\wt(y)=w \lambda$, then $\B_w(\lambda)\subseteq X$.
\end{lemma}

\begin{proof}
 Suppose $y \in X$ is an extremal element with $\wt(y)=w \lambda$. Since for any $w \in W$, we have $\B_w(\lambda) = \B_{\floor{w}^\lambda}(\lambda)$, it suffices to assume $w = \floor{w}^\lambda$. 
 
 If $\ell(w)=0$, then $w$ is the identity, $y = b_{\lambda}$, and $\B_w(\lambda) = \{b_{\lambda}\} \subseteq X$. Thus suppose $\ell(w)=n>0$, and by induction, assume for any extremal element $x \in X$ with $\wt(x) = u \lambda$ and $\ell(u) < n$, we have $\B_u(\lambda) \subseteq X$. 
 
 By Lemma~\ref{lem:path}, there exists a reduced expression $s_{j_n} \cdots s_{j_1} \in \rex(w)$ such that $y = f_{j_n}^* \cdots f_{j_1}^* (b_{\lambda})$. For any other extremal element $x\in\B_w(\lambda)-\{y\}$, we claim $x \in X$.
 In particular, by the claim, $X$ contains all extremal elements of $\B_w(\lambda)$. Since $X$, being ideal, is also extremal, it follows that $\B_w(\lambda)\subseteq X$. Thus it suffices to prove the claim, i.e. that $x \in X$.
 
 Let $\wt(x) = u\lambda$, where again we may assume $u = \floor{u}^{\lambda}$. By Lemma~\ref{lem:elements}, since $x \in \B_w(\lambda)$, we have $u \prec w$. Therefore, by the Subword Property, there exists $1 \le i_1 < \cdots < i_m \le n$ such that $s_{j_{i_m}} \cdots s_{j_{i_1}} \in \rex(u)$. Furthermore, by Lemma~\ref{lem:path}, $x = f_{j_{i_m}}^* \cdots f_{j_{i_1}}^* (b_{\lambda})$. 
 We consider two cases based on $i_m$.

 First, suppose $i_m < n$. Letting $u' = s_{j_{n-1}} \cdots s_{j_1}$, we have $u \preceq u' \prec w$. In particular, $\ell(u') < n$. Since $y \in X$ and $X$ is extremal, we may consider the extremal element $x' = e_{j_n}^*(y) \in X$ which, by Lemma~\ref{lem:path}, has $\wt(x') = u'\lambda$. Applying the inductive hypothesis to $x'$, we conclude $\B_{u'}(\lambda) \subset X$. Finally, by Lemma~\ref{lem:elements}, since $u \preceq u'$, we have $x \in \B_{u'}(\lambda) \subset X$ as desired.

 Next, suppose $i_m = n$, and let $k<n$ be the maximal omitted index for the subword for $u$.
 Let $u' = s_{j_{n}}  \cdots \hat{s}_{j_k} \cdots s_{j_1}$, that is, delete the $j_k$th term. Then $u \preceq u' \prec w$, and, in particular, $\ell(u') < n$. Since $y \in X$ and $X$ is ideal, we have 
 \[x' = f_{j_n}^* \cdots f_{j_{k+1}}^* e_{j_k}^* \cdots e_{j_n}^* (y) \in X. \]
 By Lemma~\ref{lem:path}, $x'$ is extremal with $\wt(x') = u'\lambda$.
 Applying the inductive hypothesis to $x'$, we conclude $\B_{u'}(\lambda) \subset X$. Finally, by Lemma~\ref{lem:elements}, since $u \preceq u'$, we have $x \in \B_{u'}(\lambda) \subset X$ as desired.
 \end{proof}

Since $W$ is poset under the strong Bruhat order, we may consider order ideals within $W$. 
We introduce the following definition which motivates our terminology.

\begin{definition}\label{def:idealcrystal}
    Given a lower order ideal $\mathcal{I} \subset W$, denote the subset of $\B(\lambda)$ of Demazure modules generated by $\mathcal{I}$ as follows,
    \[
    \B_{\mathcal{I}}(\lambda) = \bigcup_{w \in \mathcal{I}} \B_w(\lambda).
    \]    
\end{definition}

Since Demazure modules are always finite dimensional (see \cite[Lemma 8.1.4]{Kumar-book}), then $\B_\I(\lambda)$ is finite dimensional precisely when $\I$ is finitely generated (and thus finite).

Geometrically, the subset $\B_\I(\lambda)$ can be interpreted as follows. For $G$ the Kac-Moody group of $\g$ with Borel subgroup $B$ and $\lambda \in P^{+}$, let $\mathcal{L}_{\lambda}$ denote the line bundle on $G/B$ associated to the one-dimensional $B$-module acting with character $\lambda$.
The \emph{dual Polo module} $V_{\I}(\lambda)$ is the $B$-module dual to the space of global sections $\Gamma(\bigcup_{w \in \I} X_w,\mathcal{L}_\lambda)$ of $\mathcal{L}_\lambda$ over the union of Schubert varieties $X_w = \overline{BwB} \subset G/B$. 
Thus, it follows from Theorem~\ref{thm:kas-dem} that $\B_{\I} (\lambda)$ is the subset of $\B(\lambda)$ whose character precisely equals that of the dual Polo module $V_{\I}(\lambda)$. 

We now prove Theorem~\ref{thm:C}, showing that all ideal subsets arise in this manner, and consequently provide a local characterization for unions of Demazure crystals.

\begin{theorem}\label{thm:ideal}
    A subset $X \subseteq \B(\lambda)$ is ideal if and only if there exists a lower order ideal $\mathcal{I} \subset W$ such that $X = \B_{\I}(\lambda)$.
    Moreover, when this is the case, the ideal is given by
    \[ \mathcal{I} =  \{ \floor{w}^{\lambda} \in W \mid \wt(x) = w\lambda \ \text{for some extremal} \ x\in X \}. \]
\end{theorem}

\begin{proof}
    By Proposition~\ref{prop:ideal}, every Demazure crystal $B_w(\lambda)$ is an ideal subset of $\B(\lambda)$, and so, by Lemma~\ref{lem:ideal-union}, unions of Demazure crystals are ideal.

    Conversely, suppose $X \subset \B(\lambda)$ is an ideal subset. Define an ideal
    \[ \mathcal{I} = \{ w \in W \mid \text{some extremal } x\in X \text{ has } \wt(x) = w\lambda \}. \]
    Since $X$ is closed under raising operators, then $\mathcal{I}_X$ is a lower order ideal and, by Lemma~\ref{lem:ideal}, we have
    \[  \B_{\mathcal{I}}(\lambda) \subseteq X. \]
    To prove equality, suppose $x \in X-\B_{\mathcal{I}}(\lambda)$. Then there exist indices $i_1,\ldots,i_m$ and positive exponents $k_1,\ldots,k_m$ such that
    \[ x = f_{i_m}^{k_m} \cdots f_{i_1}^{k_1} (b_{\lambda}), \]
    where, by \cite[Prop~3.3]{ADG}, $s_{i_m} \cdots s_{i_1}$ is a reduced expression for some $w \in W$. Hence, $x \in \B_w(\lambda)$. Since $x \not\in U$, we must have $w \not\in \mathcal{I}$. However, since $X$ is extremal and $x \in X$, then we have an extremal element
    \[ y = f_{i_m}^{*} \cdots f_{i_1}^{*} (b_{\lambda}) \in X, \]
    with $\wt(y) = w \lambda$, a contradiction. Thus, $\B_{\mathcal{I}}(\lambda) = X$.
\end{proof}

It is evident that the union of ideal crystals is ideal. We note as well that the ideal property is preserved by intersections.

\begin{proposition}\label{prop:ideal-intersection}
Suppose $\lambda \in P^+$ and let $\I,\mathcal{J} \subset W$ be any lower order ideals. Then, the intersection of ideal subsets is ideal
\[
\B_{\mathcal{I}}(\lambda) \cap \B_{\mathcal{J}}(\lambda) = \B_{\mathcal{I} \cap\mathcal{J}}(\lambda).
\]
\end{proposition}

\begin{proof} 
Suppose $w,u \in W$, we will first prove that 
\[
\B_w(\lambda) \cap \B_u(\lambda) = \bigcup_{v \preceq w,u} \B_v(\lambda).
\]
By Lemma \ref{lem:elements}, if either $u \preceq \floor{w}^\lambda$ or $w\preceq \floor{u}^\lambda$ the result is trivial, so assume neither of these hold. Without loss of generality suppose $u = \floor{u}^\lambda$,$w = \floor{w}^\lambda$, and 
let $b \in \B_w(\lambda) \cap \B_u(\lambda) - \bigcup_{v \prec w,u} \B_v(\lambda)$ which, by Lemma \ref{lem:extremal-union-intersections}, we may assume is an extremal vector. So then, if $\wt(b) = v'$ we must have $b \in \B_{v'}(\lambda) \subset \B_w(\lambda)$ with $v' \preceq w$. Since $v'\not\prec w$, then necessarily $v' = w$ and thus $\B_{v'}(\lambda) = \B_w(\lambda)$. A similar deduction yields $\B_{v'}(\lambda) = \B_u(\lambda)$, a contradiction. Thus, we have that 
\[
\B_{\mathcal{I}}(\lambda) \cap \B_{\mathcal{J}}(\lambda)  = 
\bigcup_{w \in \mathcal{I}}\bigcup_{u \in \mathcal{J}}\bigcup_{v \in \langle w\rangle \cap \langle u\rangle}\B_v(\lambda)
=
\bigcup_{v \in \mathcal{I} \cap\mathcal{J} } \B_v(\lambda)=\B_{\mathcal{I} \cap\mathcal{J}}(\lambda) 
\]
which by Theorem \ref{thm:ideal} is ideal.
\end{proof}

%%%%%%%%%%%%%%%%%%%%%%%%%%%%%%%%%%%%%%%%%%%%%%%%%%%%%%%%%%%%
%  Atoms
%%%%%%%%%%%%%%%%%%%%%%%%%%%%%%%%%%%%%%%%%%%%%%%%%%%%%%%%%%%%
\section{Atomic Decompositions}

Recall that a \emph{minimal relative Schubert module} arises as the kernel of the map $\Gamma(X_w,\mathcal{L}_\lambda) \rightarrow \Gamma(\bigcup_{u \prec w} X_u,\mathcal{L}_\lambda)$. The \emph{Demazure atom} $\A_w(\lambda)$ is a certain subset of $\B(\lambda)$ whose character coincides with the dual of the minimal relative Schubert module above. 
Demazure modules and, more generally, dual Polo modules admit \emph{relative Schubert filtrations}, in which successive quotients are minimal relative Schubert modules \cite{vdk}.
Thus a $B$-module with a Schubert filtration, where successive quotients are Polo modules, necessarily also has a relative Schubert filtration. 
For $\g$ finite, this was proven geometrically in \cite[Lemma 2.2]{vdk} and crystal-theoretically in \cite[Theorem 3.4]{Armon25}. 

Characters of Demazure atoms were also considered by Lascoux and Sch\"utzenberger \cite{LaSchu90keys} under the term \emph{standard bases}.
Atomic decompositions have historically and continue to be widely studied both algebraically and combinatorially; see for instance \cite{ mason, ChoiKwon, Santos, JaconLecouvey,BucimasScrimshaw, brubaker,BSW, Blasiak, Armon25}.

Below, we show that any ideal subset decomposes as a disjoint union of Demazure atoms, thus extending Armon's result in \cite{Armon25} from complex semisimple Lie algebras to all symmetrizable Kac-Moody Lie algebras.

\begin{definition}\label{def:atom}
    For $\lambda \in P^+$ and $w \in W$, the \newword{Demazure atom} $\A_{w}(\lambda)$ is the subset of $\B(\lambda)$ given by, 
    \[
    \A_w(\lambda) = \B_w(\lambda) - \bigcup_{v\prec \floor{w}^{\lambda}} \B_v(\lambda).
    \]
\end{definition}

Unlike all other subsets considered in this paper, Demazure atoms are not generally extremal. However, they do have the following string property.

\begin{proposition}
    If $x \in \A_w(\lambda)$ and $e_i(x) \ne 0$, then $f_i^k(x) \in \A_w(\lambda) \sqcup \{0\}$ for all $k>0$.
\end{proposition}

\begin{proof}
    Without lose of generality, we assume $w = \floor{w}^{\lambda}$. If $x\in \A_w(\lambda) \subset \B_w(\lambda)$, then since $\B_w(\lambda)$ is extremal by Proposition~\ref{prop:extremal}, we have $e_i(x) \in \B_w(\lambda)$. Therefore $s_i w \prec w$, and so $\B_w(\lambda) = \F_i (\B_{s_i w}(\lambda))$. The claim now follows from the definition of $\A_w(\lambda)$.
\end{proof}

It follows from the definition that any Demazure crystal is a union of Demazure atoms. In fact, the union of atoms gives a disjoint decomposition.

\begin{lemma}\label{lem:atom-disjoint}
    For $\lambda\in P^{+}$ and $w,v \in W$, if $\floor{w}^{\lambda} \ne \floor{v}^{\lambda}$, then we have $\A_w(\lambda) \cap \A_v(\lambda) = \varnothing$. In particular,  
    \begin{equation}\label{eq:atomic}
        \B_w(\lambda) = \bigsqcup_{\substack{\floor{u}^\lambda \preceq w}} \A_{\floor{u}^\lambda}(\lambda).
    \end{equation}
\end{lemma}

\begin{proof}
    Without loss of generality, we may assume $w = \floor{w}^{\lambda}$. Let $x \in \A_w(\lambda) \subset \B_w(\lambda)$. Since $x \in \B_w(\lambda)$, we may write 
    \[ x = f_{j_n}^{k_n} \cdots f_{j_1}^{k_1} (b_{\lambda}) \]
    where each $k_i > 0$ and $s_{j_1} \cdots s_{j_n}$ is a subword of some $\rho \in \rex(w)$. However, since $x \in \A_w(\lambda)$ implies $x \not\in\B_u(\lambda)$ for any $u \prec w$, we also have $s_{j_1} \cdots s_{j_n} \not\in \rex(u)$ for any $u \prec w$. Therefore, $s_{j_1} \cdots s_{j_n} \in \rex(w)$. In particular, any minimal length path from $b_\lambda$ to $x$ must have length $n$ and give rise to a reduced expression for $w$.
    
    Therefore, if $x \in \A_w(\lambda) \cap \A_v(\lambda)$, where again we assume $v = \floor{v}^{\lambda}$, then by the same argument, we have $s_{j_1} \cdots s_{j_n} \in \rex(v)$. Thus $w=v$.
\end{proof}

Finally, we prove the crystal-theoretic version of the fact that Polo modules admit relative Schubert filtrations.

\begin{theorem}\label{thm:atom-positive}
    For any lower order ideal $\mathcal{I} \subset W$ and $\lambda \in P^+$, the ideal subset $\B_{\mathcal{I}}(\lambda)$ decomposes into Demazure atoms as, 
    \[\B_{\mathcal{I}}(\lambda) = \bigsqcup_{v \in \mathcal{I}} \A_v(\lambda).\] 
\end{theorem}

\begin{proof}
This follows directly from Lemma \ref{lem:atom-disjoint} since by definition, 
\[
\B_\I(\lambda) = \bigcup_{w \in \I} \B_w(\lambda) =
\bigcup_{w \in \I}
\bigsqcup_{\substack{\floor{v}^\lambda \preceq w}} \A_{\floor{v}^\lambda}(\lambda)=
\bigcup_{v \in \I} \A_v(\lambda) = \bigsqcup_{v \in \I} \A_v(\lambda).
\]
\end{proof}

Thus, combining Theorem \ref{thm:atom-positive} with Proposition~\ref{prop:ideal-intersection} yields the following decomposition of intersections of ideal subsets.

\begin{corollary}\label{cor:Dem-atom-intersection}
    For any $\lambda \in P^+$ and lower ideals $\I, \mathcal{J} \subset W$, we have
    \[
\B_\I(\lambda) \cap \B_{\mathcal{J}}(\lambda) = \bigsqcup_{v \in \I \cap \mathcal{J}} \A_v(\lambda).
    \]
\end{corollary}

In particular, for $\g = \mathfrak{sl}_n$, Theorem~\ref{thm:atom-positive} asserts that the character of a Polo module is a nonnegative sum of Demazure atom polynomials. 

%%%%%%%%%%%%%%%%%%%%%%%%%%%%%%%%%%%%%%%%%%%%%%%%%%%%%%%%%%%%
%  Demazure
%%%%%%%%%%%%%%%%%%%%%%%%%%%%%%%%%%%%%%%%%%%%%%%%%%%%%%%%%%%%
\section{Demazure subsets of highest weight crystals} 

We now return to Demazure crystals and prove Theorems \ref{thm:A} and \ref{thm:B}.
We begin with an immediate corollary to Theorem~\ref{thm:ideal} characterizing Demazure crystals.

\begin{corollary}
    A subset $X \subset \B(\lambda)$ is a Demazure crystal if and only if $X$ equals some ideal subset $\B_\I(\lambda)$ generated by a principal lower order ideal $\I$.
\end{corollary}

Equivalently, we can use the following condition to determine when an ideal subset is generated by a principal ideal.

\begin{definition}\label{def:orderly}
  An extremal subset $X \subseteq \B(\lambda)$ is \newword{principal} if there exists an extremal  $x \in X$ such that for all extremal  $y\in X$, we have $\wt(x) \preceq \wt(y)$.
\end{definition}

It is straightforward to show that Demazure subsets are principal.

\begin{proposition}\label{prop:orderly}
    For any $\lambda\in P^{+}$ and $w\in W$, the Demazure crystal $\B_w(\lambda) \subseteq \B(\lambda)$ is a principal subset.
\end{proposition}

\begin{proof}
    By Proposition~\ref{prop:extremal}, $\B_w(\lambda)$ is an extremal subset. Suppose $y \in \B_w(\lambda)$ is extremal, say with $\wt(y) = u\lambda$ for some $u = \floor{u}^{\lambda}$. By Proposition~\ref{prop:ideal}, $\B_w(\lambda)$ is an ideal subset, and so by Lemma~\ref{lem:ideal}, $\B_u(\lambda) \subseteq \B_w(\lambda)$.
    Therefore, by Lemma~\ref{lem:elements}, $u = \floor{u}^{\lambda} \preceq w$. In particular, $\wt(x) \preceq \wt(y)$ for $x\in\B_w(\lambda)$ extremal of weight $w$.
\end{proof}

Both the ideal and principal conditions are necessary for Demazure subsets, though neither alone is enough to ensure that a subset is Demazure. For example, for $\mathfrak{sl}_3$, consider the highest weight crystal $\B(\lambda)$, where $\lambda = \omega_1 + \omega_2$ is the sum of the first two fundamental weights. The subset
\[ X_1 = \{ b_{\lambda}, f_1(b_{\lambda}), f_2(b_{\lambda})\} \cong \B_{s_1}(\lambda) \cup \B_{s_2}(\lambda) \]
is ideal but not principal, and the subset
\[ X_2 = \{ b_{\lambda}, f_1(b_{\lambda}), f_2 f_1 (b_{\lambda}), f_2^2 f_1 (b_{\lambda}) \} \]
is principal but not ideal. Taken together, however, these two conditions are sufficient to ensure a subset is a Demazure subset.

Together with Theorem \ref{thm:ideal}, the following establishes Theorem \ref{thm:A}.

\begin{theorem}\label{thm:demazure}
    A subset $X \subseteq \B(\lambda)$ is ideal and principal if and only if $X = \B_w(\lambda)$ for some $w \in W$.
\end{theorem}

\begin{proof}
    Propositions~\ref{prop:ideal} and \ref{prop:orderly} show that $\B_w(\lambda)$ is ideal and principal.
    
    Conversely, suppose $X$ is an ideal and principal subset of $\B(\lambda)$. 
    By Theorem~\ref{thm:ideal}, we have
    \[ X = \bigcup_{u \in \mathcal{I}} \B_u(\lambda), \]
    where the lower order ideal $\mathcal{I} \subset W$ is given by
    \[ \mathcal{I} = \bigcup_{\substack{y \ \text{extremal} \\ \wt(y) = u\lambda}} \{ \floor{u}^{\lambda} \}. \]  
    Since $X$ is principal, there exists $x\in X$ with $\wt(x) = w\lambda$ such that for all extremal $y\in X$ with $\wt(y) = u\lambda$, we have $\floor{u}^{\lambda} \preceq w$. In particular, for each $u \in \mathcal{I}$, we have $u \preceq w$, and so since $w\in \mathcal{I}$, we conclude that $\mathcal{I} = \langle w \rangle$ is principal. Furthermore, by Lemma~\ref{lem:elements}, for $u\in\mathcal{I}$, we have $\B_u(\lambda) \subseteq \B_w(\lambda)$. Therefore $X = \B_w(\lambda)$ is a Demazure subset. 
\end{proof}

In particular, if an ideal subset $\B_{\I}(\lambda)$ contains an element of extremal weight $w\lambda$, say with $w = \floor{w}^{\lambda}$, then $w \in \I$. Therefore,
\[ \B_{\I}(\lambda) = \B_w(\lambda) \sqcup Y, \]
where $Y$ is empty if and only if $\I$ is principally generated by $w$. Thus the  character of $\B_{\I}(\lambda)$ is equal to the character of $\B_w(\lambda)$ if and only if $\I$ is principally generated by $w$. This establishes Theorem~\ref{thm:B}. That is, the character of an ideal crystal determines whether it is a Demazure crystal. 

%%%%%%%%%%%%%%%%%%%%%%%%%%%%%%%%%%%%%%%%%%%%%%%%%%%%%%%%%%%%
%  Acknowledgments
%%%%%%%%%%%%%%%%%%%%%%%%%%%%%%%%%%%%%%%%%%%%%%%%%%%%%%%%%%%%

\subsection*{Acknowledgments} We thank Dora Woodruff and Rosa Paten for prompting us to characterize Demazure crystals in general type.

%%%%%%%%%%%%%%%%%%%%%%%%%%%%%%%%%%%%%%%%%%%%%%%%%%%%%%%%%%%%
%  Bibliography
%%%%%%%%%%%%%%%%%%%%%%%%%%%%%%%%%%%%%%%%%%%%%%%%%%%%%%%%%%%%

\bibliographystyle{plain}
\bibliography{main}

\end{document}